\documentclass[11pt]{amsart}
\usepackage{etex}

\usepackage{amsmath,amssymb,amsthm,mathtools}
\usepackage[margin=1.12in]{geometry}
\usepackage[colorlinks=true,linkcolor=blue,citecolor=blue,urlcolor=blue]{hyperref}
\newtheorem{theorem}{Theorem}[section]
\newtheorem{proposition}[theorem]{Proposition}
\newtheorem{lemma}[theorem]{Lemma}
\newtheorem{corollary}[theorem]{Corollary}
\theoremstyle{definition}
\newtheorem{definition}[theorem]{Definition}
\theoremstyle{remark}
\newtheorem{remark}[theorem]{Remark}

\newcommand{\C}{\mathbb{C}}
\newcommand{\Z}{\mathbb{Z}}
\newcommand{\cO}{\mathcal{O}}
\newcommand{\cH}{\mathcal{H}}
\newcommand{\cS}{\mathcal{S}}
\newcommand{\cA}{\mathcal{A}}
\newcommand{\cF}{\mathcal{F}}
\newcommand{\GL}{\mathrm{GL}}
\newcommand{\Tr}{\operatorname{Tr}}
\newcommand{\diag}{\operatorname{diag}}
\newcommand{\vp}{\varpi}
\newcommand{\tw}{\mathrm{tw}}
\newcommand{\HG}{\mathcal{H}_{G'}}
\newcommand{\HS}{\mathcal{H}_{S}}
\newcommand{\HOff}{\mathcal{H}^{\mathrm{Off}}}

\title[A local proof of Offen's spherical transfer theorem]
{A Local Fourier-Transform Proof of\\
Offen's Spherical Transfer Theorem}
\author{Zhengyu Mao}
\address{Department of Mathematics and Computer Science, Rutgers University, Newark, NJ 07102, USA}
\email{zmao@rutgers.edu}
\author{Chen Wan}
\address{Department of Mathematics and Computer Science, Rutgers University, Newark, NJ 07102, USA}
\email{chen.wan@rutgers.edu}
\date{}

\begin{document}

\begin{abstract}
Let $E/F$ be an unramified quadratic extension of non-archimedean local fields of characteristic zero and odd residual characteristic.  Offen proved that, on the full spherical space of nondegenerate Hermitian matrices, equality of the Hironaka and Satake transforms implies Jacquet--Ye matching; see \cite[Theorem~10.1]{O09} and the formulation in \cite[Theorem~8.1(3)]{FLO12}.  We give a purely local proof of this theorem.  The argument combines Jacquet's compatibility of matching with Fourier transform, the Godement--Jacquet functional equation, and Hironaka's scalar functional equation for Hermitian zeta integrals.  We also make explicit the normalization dictionary between the sign convention used in Offen's rank-two paper \cite{O06} and that used in his all-rank theorem \cite{O09}.  The local proof recovers Jacquet's spherical-Hecke fundamental lemma from the unit case.
\end{abstract}

\maketitle

\section{Introduction}

\subsection{Background and the known theorem}
Let $E/F$ be an unramified quadratic extension of non-archimedean local fields of characteristic zero and odd residual characteristic.  Jacquet's relative trace formula compares Kloosterman orbital integrals on $G'=\GL_n(F)$
with Kloosterman orbital integrals on the space $S$ of invertible Hermitian matrices.  The fundamental lemma for the unit element was proved by Jacquet \cite{J1} (and by Ng\^o \cite{Ngo} in positive characteristic), and the fundamental lemma for the spherical Hecke algebra was proved by Jacquet \cite{J2}.  In the latter theorem the Hecke algebra $\cH_G$ of $G=\GL_n(E)$ acts on the two sides through base change
\[
b:\cH_G\longrightarrow\HG
\]
on the $\GL_n(F)$ side and through convolution on the Hermitian side.

The full spherical space
\[
\HS=C_c^\infty(S)^K,\qquad K=\GL_n(\cO_E),
\]
is much larger: Hironaka showed that it is a free $\cH_G$-module of rank $2^n$ \cite{H99}.  Jacquet's spherical-Hecke theorem (proved by a global argument) concerns the cyclic submodule generated by $\mathbf 1_{K\cap S}$.  Using Hironaka's spherical transform and the Satake isomorphism, one can establish a canonical bijection between $\HS$ and $\HG$; Offen \cite{O06} formulated the stronger assertion that arbitrary spherical functions under this bijection should match, and proved it directly for $n=2$. Later Offen proved the result for general $n$ in \cite[Theorem~10.1]{O09} by a global argument based on a simple relative trace formula.  Feigon--Lapid--Offen later recorded the result in exactly the transform-theoretic form used here: equality of the Hironaka and Satake transforms implies matching \cite[Theorem~8.1(3)]{FLO12}.  Beuzart-Plessis \cite{BP22} subsequently developed the local Plancherel theory of the Hermitian space and the associated relative-character formalism.  

The purpose of this note is to provide an entirely local proof of Offen's theorem. We believe this local approach is of its independent interest, and will in future work develop it to further generality.

\subsection{The Theorem}

Hironaka's transform \cite{H99} is naturally written in a spectral variable $z=(z_1,\dots,z_n)$.  To compare it with the Satake transform one chooses
\[
x_i=\epsilon q^{-z_i},\qquad \epsilon\in\{\pm1\}.
\]
We write
\[
\cH_\epsilon:\HS\longrightarrow
\Lambda:=\C[x_1^{\pm1},\dots,x_n^{\pm1}]^{S_n}
\]
for Hironaka's transform (which is a bijection).  Both $\epsilon$ normalizations are $\cH_G$-linear because base change acts through the squares $x_i^2$.  The two choices differ by the determinant twist $\omega\circ\det$; see Lemma~\ref{lem:change}.

For $\delta\in\{0,1\}$ set
\[
\gamma_\delta(a)=\prod_{k=1}^n\omega(a_k)^{\delta+k},
\qquad a=\diag(a_1,\dots,a_n)\in \GL_n(F),
\]
where $\omega$ is the quadratic character of $F^\times$ attached to $E/F$.

We use $\Omega'[\Phi,\psi;a]$ and $\Omega[\Psi,\psi;a]$ to denote the Kloosterman orbital integrals on $G'$ and on $S$ respectively. Here $a=\diag(a_1,\dots,a_n)\in \GL_n(F)$ and $\psi$ is an unramified additive character. Let $\cS$ be the Satake transform on the Hecke algebra $G'$ (which is an isomorphism to $\Lambda$).

\begin{theorem}[Offen's spherical transfer theorem]\label{thm:offen}
Let $n\ge1$ and $\epsilon\in\{\pm1\}$, and put
\[
\delta_\epsilon=\frac{1-\epsilon}{2}.
\]
If $\Phi\in\HG$ and $\Psi\in\HS$ satisfy
\[
\cS\Phi=\cH_\epsilon\Psi,
\]
then, for every invertible diagonal matrix $a$,
\[
\Omega'[\Phi,\psi;a]
 =\gamma_{\delta_\epsilon}(a)\,\Omega[\Psi,\psi;a]
 =\epsilon^{v(\det a)}\prod_{k=1}^n\omega(a_k)^k\,\Omega[\Psi,\psi;a].
\]
\end{theorem}

\begin{remark}[Equivalent normalizations]\label{cor:norm}
In the normalization $x_i=q^{-z_i}$ used in \cite[Theorem~10.1]{O09}, the transfer factor is $\gamma_0$.  In the normalization $x_i=-q^{-z_i}$ used in the explicit rank-two formulas of \cite{O06}, the transfer factor is $\gamma_1$.  The two formulations are equivalent by twisting with $\omega\circ\det$.

Our proof will work in the $\epsilon=-1$ normalization; Lemma~\ref{lem:change} then recovers the $\epsilon=+1$ statement.
\end{remark}

\subsection{The local proof in outline}

When $\cS\Phi=\cH_\epsilon\Psi=P\in \Lambda$ we call $(\Phi,\Psi)$ a Hironaka pair. We say the pair matches if $\Phi$ and $\Psi$ have matching orbital integrals with transfer factor $\gamma_1$.
Let $V\subset\Lambda$ be the set of spherical transforms $P$ for which the corresponding Hironaka pair  matches. The key result is that there is an algebra of multipliers $\cA\subset \Lambda$ such that when $D\in \cA$, $P\in V$ we have $DP\in V$. Since the unit Hecke element matching gives $1\in V$, the proof is completed when we show $\cA=\Lambda$.

To construct the multipliers, we rely on two inputs. One is Jacquet's result that the property of having matching orbital integrals is preserved under Fourier transform. The second is the observation that Hironaka pair is preserved under Fourier transform. For this result we need to compare Godement-Jacquet function equation with Hironaka's function equation, and observe that the same Gamma factors appear in these function equations. Therefore one can use Fourier transform to construct new matching Hironaka pairs (though we need to pay attention to compactness issue.)

Finally with the simple observation that $V$ is graded, we can apply Fourier transform to matching homegeneous Hironaka pairs and produce many multipliers (Proposition~6.1).

\subsection{Organization}
Section~\ref{sec:setup} fixes notation and the two spherical transforms.  Section~3 records elementary properties of $V$.  Section~4 computes Fourier transform in spherical coordinates on both sides.  Section~5 recalls Jacquet's compatibility of matching with Fourier transform.  Section~\ref{sec:proof} combines these inputs and the Hall--Littlewood generation argument to give the local proof.

\subsection{Acknowledgment} The argument in this paper was generated largely by Claude Opus, with additional contributions from GPT. The authors’ role consisted primarily of guiding the AI systems toward a local proof that does not rely on the existing proof in the literature, preparing the manuscript, and carefully checking the accuracy of the arguments. We thank H. Jacquet for his encouragement and his insistence on finding a simpler proof of the fundamental lemma.

The second author’s work is partially supported by the NSF grant DMS-2349836 and a Simons Travel Grant.

\section{Notation and spherical transforms}\label{sec:setup}

\subsection{Fields and groups}
Let $\cO_F,\cO_E$ be the rings of integers, $\vp\in F$ a common uniformizer, $q$ the cardinality of the residue field of $F$, and $|\cdot|$ the normalized absolute value of $F$. Let $\omega$ be the unramified quadratic character of $F^\times$, so $\omega(\vp)=-1$. Fix an additive character $\psi$ of $F$ of conductor $\cO_F$.

Put $G=\GL_n(E)$, $G'=\GL_n(F)$, $K=\GL_n(\cO_E)$, $K'=\GL_n(\cO_F)$. Let $V_n$ be the space of Hermitian matrices ${}^t\bar x=x$ in $M_n(E)$ and $S=V_n\cap G$. The group $G$ acts on $S$ by $g\cdot s=gs\,{}^t\bar g$. Let $\HG$ be the spherical Hecke algebra of $G'$ and $\HS$ the space of compactly supported $K$-invariant functions on $S$. For a dominant $\lambda\in\Z^n$ ($\lambda_1\ge\dots\ge\lambda_n$) put $\vp^\lambda=\diag(\vp^{\lambda_1},\dots,\vp^{\lambda_n})$,
\[
c_\lambda=\mathbf 1_{K'\vp^\lambda K'},\qquad C_\lambda=\mathbf 1_{K\cdot\vp^\lambda},
\]
\[
|\lambda|=\sum_i\lambda_i,\qquad n(\lambda)=\sum_i(i-1)\lambda_i,\qquad \langle\rho,\lambda\rangle=\sum_i\tfrac{n+1-2i}{2}\lambda_i=\tfrac{n-1}2|\lambda|-n(\lambda).
\]
The $c_\lambda$ form a basis of $\HG$ and the $C_\lambda$ a basis of $\HS$ \cite{H99}. Write $\tilde\lambda=(-\lambda_n,\dots,-\lambda_1)$.

\subsection{Orbital integrals and matching}
Let $\Phi\in C_c^\infty(M_n(F))$ and $\Psi\in C_c^\infty(V_n)$.  For $a=\diag(a_1,\dots,a_n)$ with $a_1\cdots a_{n-1}\ne0$, put ($N', N$ are the groups of upper triangular matrices with unit diagonal in $G'$ and $G$ respectively)
\[
\Omega'[\Phi,\psi;a]=\int_{N'\times N'}\Phi({}^tu_1au_2)\,\theta(u_1u_2)\,du_1\,du_2,\qquad
\Omega[\Psi,\psi;a]=\int_{N}\Psi({}^t\bar uau)\,\theta(u\bar u)\,du,
\]
where $\theta(u)=\psi(\sum_iu_{i,i+1})$, $\theta(u\bar u)=\psi(\sum_i(u_{i,i+1}+\bar u_{i,i+1}))$, and the Haar measures give $N'(\cO_F)$ and $N(\cO_E)$ volume one.

Let
\[
\gamma_J(a)=\prod_{k=1}^{n-1}\omega(a_1\cdots a_k)
\]
be Jacquet's transfer factor, which is defined whenever
$a_1\cdots a_{n-1}\ne0$. For invertible diagonal $a$, comparison
of exponents gives
\begin{equation}\label{eq:gammas}
	\gamma_\delta(a)=\gamma_J(a)\,\omega(\det a)^{n+\delta}.
\end{equation}
We say that $\Phi$ and $\Psi$ $\delta$-match if
\[
\Omega'[\Phi,\psi;a]
=\gamma_\delta(a)\,\Omega[\Psi,\psi;a]
\]
for every invertible diagonal matrix $a$. We say that they match
in Jacquet's sense if
\[
\Omega'[\Phi,\psi;a]
=\gamma_J(a)\,\Omega[\Psi,\psi;a]
\]
for every diagonal $a$ with $a_1\cdots a_{n-1}\ne0$, allowing
$a_n=0$.

For a function $\Psi$ on $S$, define
\[
\Psi^{\mathrm{tw}}(x)
=\omega(\det x)^{n+1}\Psi(x),\qquad x\in S.
\]
Since the determinant is constant on the orbits, the displayed
identity between transfer factors shows that $\Phi$ and $\Psi$
$1$-match if and only if
\[
\Omega'[\Phi,\psi;a]
=\gamma_J(a)\,\Omega[\Psi^{\mathrm{tw}},\psi;a]
\]
for every invertible diagonal $a$.

When $\Psi\in H_S$, both $\Psi$ and $\Psi^{\mathrm{tw}}$ are
extended by zero to $V_n$; these extensions are Schwartz functions.
We do not assume that determinant twisting preserves arbitrary
Schwartz functions on $V_n$.

\subsection{Spherical transforms}
Let $\Lambda=\C[x_1^{\pm1},\dots,x_n^{\pm1}]^{S_n}=\C[e_1,\dots,e_{n-1},e_n^{\pm1}]$, where $e_r$ denotes the elementary symmetric polynomials. Let $P_\lambda(x;t)$ be the Hall--Littlewood polynomials \cite[III]{Mac}. The Satake isomorphism $\cS\colon\HG\to\Lambda$ is normalized by
\[
\cS(c_\lambda)=q^{\langle\rho,\lambda\rangle}P_\lambda(x;q^{-1})
\]
\cite[V]{Mac}. On the Hermitian side we use Hironaka's spherical functions \cite{H99}
\[
\omega(x;s)=\int_K\prod_{i=1}^n|d_i(k\cdot x)|^{s_i}\,dk,
\]
where $d_i$ is the $i$-th leading principal minor, written in Hironaka's variable $z\in\C^n$:
\begin{equation}\label{eq:sz}
s_i=-z_i+z_{i+1}-1-\tfrac{\pi\sqrt{-1}}{\log q}\ (1\le i\le n-1),\qquad s_n=-z_n+\tfrac{n-1}2-\tfrac{\pi\sqrt{-1}}{\log q}.
\end{equation}
(In \cite{H98} the sign of $\pi\sqrt{-1}/\log q$ is reversed, which changes nothing since only $q^{z}$ matters.) Put $\Psi_z(x)=\omega(x;z)/\omega(1_n;z)$ and let $dx$ be the $G$-invariant measure on $S$ with $\int_{K\cdot1_n}dx=1$. Hironaka's transform is $\varphi\mapsto\int_S\varphi(x)\Psi_z(x^{-1})\,dx$.

\begin{definition}
For $\epsilon\in\{\pm1\}$ and $\varphi\in\HS$, let $\cH_\epsilon\varphi\in\Lambda$ be Hironaka's transform of $\varphi$ written in the variable
\begin{equation}\label{eq:xz}
x_i=\epsilon\,q^{-z_i}\qquad(1\le i\le n).
\end{equation}
\end{definition}

By \cite[Theorem 2]{H99}, $\cH_\epsilon\colon\HS\to\Lambda$ is a linear isomorphism. For either $\epsilon$ it is $\cH_G$-linear, where $f\in\cH_G$ acts on $\Lambda$ by multiplication by the Satake transform of $f$ evaluated at $x^2=(x_1^2,\dots,x_n^2)$, i.e. through base change: by \cite[(2.3)]{H99} the Hecke eigenvalues are functions of $q^{-2z}=x^2$.

\begin{lemma}\label{lem:exact}
For every dominant $\lambda$,
\[
\cH_\epsilon(C_\lambda)(x)=\epsilon^{|\lambda|}(-1)^{n(\lambda)+n|\lambda|}\,q^{\langle\rho,\lambda\rangle}\,P_\lambda(x;-q^{-1}).
\]
\end{lemma}

\begin{proof}
By definition $\cH_\epsilon(C_\lambda)=v(K\cdot\vp^\lambda)\,\Psi_z(\vp^{-\lambda})$, and $\vp^{-\lambda}\in K\cdot\vp^{\tilde\lambda}$. By \cite[(2.6), (2.7)]{H99},
\begin{align*}
\Psi_z(\vp^\mu)
 &=(-1)^{n(\mu)+|\mu|}
   q^{n(\mu)-\frac{n-1}{2}|\mu|}
   \frac{w_\mu(-q^{-1})}{w_n(-q^{-1})}
   P_\mu(q^{z};-q^{-1}),\\
 v(K\cdot\vp^\lambda)
 &=q^{-2n(\lambda)+(n-1)|\lambda|}
   \frac{w_n(-q^{-1})}{w_\lambda(-q^{-1})}.
\end{align*}
with $w_\mu=w_\lambda$ for $\mu=\tilde\lambda$. Using $|\tilde\lambda|=-|\lambda|$, $n(\tilde\lambda)=n(\lambda)-(n-1)|\lambda|$ and $P_{\tilde\lambda}(y;t)=P_\lambda(y^{-1};t)$, the power of $q$ is $q^{\langle\rho,\lambda\rangle}$ and the sign is $(-1)^{n(\lambda)-n|\lambda|}$, while $P_\lambda(q^{-z};-q^{-1})=P_\lambda(\epsilon x;-q^{-1})=\epsilon^{|\lambda|}P_\lambda(x;-q^{-1})$.
\end{proof}

\begin{lemma}\label{lem:change}
Let $\Psi\in\HS$ and $\Phi\in\HG$.
\begin{enumerate}
\item $\cH_{-\epsilon}(\Psi)=\cH_\epsilon(\Psi\cdot\omega\circ\det)$.
\item $\Phi$ and $\Psi$ $\delta$-match if and only if $\Phi$ and $\Psi\cdot\omega\circ\det$ $(1-\delta)$-match.
\end{enumerate}
Consequently, Offen's spherical-transfer theorem in the $\epsilon$ normalization is equivalent to the theorem in the $-\epsilon$ normalization.
\end{lemma}

\begin{proof}
(1) By definition $\cH_{-\epsilon}(\Psi)(x)=\cH_\epsilon(\Psi)(-x)$. For $\Psi=C_\lambda$ the right side is $(-1)^{|\lambda|}\cH_\epsilon(C_\lambda)(x)$ by homogeneity, and $(-1)^{|\lambda|}$ is the value of $\omega\circ\det$ on $K\cdot\vp^\lambda$. (2) Since $\det$ is constant on orbits, $\Omega[\Psi\cdot\omega\circ\det;a]=\omega(\det a)\Omega[\Psi;a]$, while $\gamma_{1-\delta}(a)=\gamma_\delta(a)\,\omega(\det a)$ because $\omega^2=1$. For the last assertion, if $\cS\Phi=\cH_{-\epsilon}\Psi$ then $\cS\Phi=\cH_\epsilon(\Psi\cdot\omega\circ\det)$ by (1); by (2), $\Phi$ and $\Psi\cdot\omega\circ\det$ $\delta$-match exactly when $\Phi$ and $\Psi$ $(1-\delta)$-match, and $1-\delta_\epsilon=\delta_{-\epsilon}$.
\end{proof}

\begin{remark}\label{rem:offen}
There are two useful comparisons with Offen's papers.

First, for $n=2$ and $\lambda=(m,k)$, Lemma~\ref{lem:exact} gives
\[
\cH_{-1}(C_{(m,k)})=(-1)^m q^{(m-k)/2}P_{(m,k)}(x;-q^{-1}),
\]
which is the formula of \cite{O06}; thus the direct rank-two calculation there uses $\epsilon=-1$ and transfer factor $\gamma_1$.

Second, write $\HOff$ for the all-rank spherical transform used by Offen in \cite[\S6, especially (6.14)]{O09}.  With the change of variables $x=Z^{-1}$, his formula on the orbit basis is
\[
\HOff(C_\lambda)(x)=(-1)^{n(\lambda)+n|\lambda|}
q^{\langle\rho,\lambda\rangle}P_\lambda(x;-q^{-1}).
\]
Comparison with Lemma~\ref{lem:exact} shows that, in the present notation,
\[
\HOff=\cH_{+1}.
\]
The transfer factor in \cite[Theorem~10.1]{O09} is correspondingly $\gamma_0(a)=\prod_k\omega(a_k)^k$.  Lemma~\ref{lem:change} converts this statement into the equivalent $\cH_{-1}$, $\gamma_1$ formulation.  Thus the apparent difference between the 2006 and 2009 signs is a normalization change, not a difference in the transfer theorem.
\end{remark}

\subsection{The working normalization}
From now on, until the local proof in Section~\ref{sec:proof}, we take $\epsilon=-1$ and write $\cH=\cH_{-1}$. Lemma~\ref{lem:exact} then reads
\begin{equation}\label{eq:exactminus}
\cH(C_\lambda)(x)=\sigma(\lambda)\,q^{\langle\rho,\lambda\rangle}\,P_\lambda(x;-q^{-1}),\qquad \sigma(\lambda)=(-1)^{n(\lambda)+(n+1)|\lambda|}.
\end{equation}
A pair $(\Phi,\Psi)\in\HG\times\HS$ with $\cS\Phi=\cH\Psi$ is called a \emph{Hironaka pair}. Let
\[
V=\{P\in\Lambda:\ \text{the Hironaka pair with transform $P$ $1$-matches}\}.
\]
This is a linear subspace. For $\epsilon=-1$, the spherical-transfer theorem is equivalent to the statement $V=\Lambda$.

\subsection{Shells}
For a function $f$ on $M_n(F)$ or $V_n$ and $m\in\Z$ put $f^{(m)}=f\cdot\mathbf 1_{\{v(\det)=m\}}$. If $f$ is a bi-$K'$-invariant (resp. $K$-invariant) Schwartz function, each $f^{(m)}$ lies in $\HG$ (resp. $\HS$), and $f^{(m)}=0$ for $m$ small. We extend $\cS$ and $\cH$ to such $f$ as formal sums
\[
\cS f=\sum_m\cS(f^{(m)}),\qquad \cH f=\sum_m\cH(f^{(m)}),
\]
where $\cS(f^{(m)})$ and $\cH(f^{(m)})$ are homogeneous of degree $m$, since $c_\lambda,C_\lambda$ have degree $|\lambda|$.

\section{Elementary properties of \texorpdfstring{$V$}{V}}

\begin{lemma}\label{lem:graded}
$V$ is graded: $P\in V$ if and only if every homogeneous component of $P$ lies in $V$.
\end{lemma}

\begin{proof}
The homogeneous component of degree $m$ of the transform of $(\Phi,\Psi)$ is the transform of $(\Phi^{(m)},\Psi^{(m)})$, and $\Omega'[\Phi^{(m)};a]$ is the restriction of $\Omega'[\Phi;a]$ to $\{v(\det a)=m\}$, and similarly for $\Psi$.
\end{proof}

\begin{lemma}\label{lem:unit}
$1\in V$.
\end{lemma}

\begin{proof}
The Hironaka pair with transform $1$ is $(c_0,C_0)=(\mathbf 1_{K'},\mathbf 1_{K\cap S})$. These match in Jacquet's sense: restrict the fundamental lemma of \cite{J1} for the characteristic functions of $M_n(\cO_F)$ and of $V_n\cap M_n(\cO_E)$ to the shell $v(\det)=0$, as in \cite[\S1]{J2}. On this shell $\gamma_1=\gamma_J$ by \eqref{eq:gammas}.
\end{proof}

\begin{lemma}\label{lem:inversion}
$V$ is stable under $P(x)\mapsto P(x^{-1})$, and under $\iota\colon P(x)\mapsto P(q^{-n}x^{-1})$.
\end{lemma}

\begin{proof}
Let $w=w_n$ be the longest Weyl element and put $\check\Phi(g)=\Phi(g^{-1})$, $\check\Psi(s)=\Psi(s^{-1})$. Write $u_2^{-1}a^{-1}\,{}^tu_1^{-1}=w\,(wu_2^{-1}w)(wa^{-1}w)(w\,{}^tu_1^{-1}w)\,w$. Since $wu_2^{-1}w={}^tn_1$ and $w\,{}^tu_1^{-1}w=n_2$ with $n_1,n_2\in N'$, $\theta(n_1)=\overline{\theta(u_2)}$, $\theta(n_2)=\overline{\theta(u_1)}$, and $\Phi$ is bi-$K'$-invariant, we get
\[
\Omega'[\check\Phi,\psi;a]=\Omega'[\Phi,\bar\psi;wa^{-1}w]=\Omega'[\Phi,\psi;wa^{-1}w];
\]
the last equality follows by conjugating with $d=\diag(1,-1,1,\dots)\in K'$, which commutes with $a$. The same computation, with $n=w\,{}^t\bar u^{-1}w\in N$, so that $wu^{-1}w={}^t\bar n$ and $\theta(n\bar n)=\overline{\theta(u\bar u)}$, gives $\Omega[\check\Psi,\psi;a]=\Omega[\Psi,\psi;wa^{-1}w]$. One checks $\gamma_1(wa^{-1}w)=\gamma_1(a)\,\omega(\det a)^{n+1}$. Hence if $(\Phi,\Psi)$ $1$-match, so do $(\check\Phi,\check\Psi^\tw)$.

On transforms, $\check c_\lambda=c_{\tilde\lambda}$ and $\cS(c_{\tilde\lambda})(x)=\cS(c_\lambda)(x^{-1})$, since $\langle\rho,\tilde\lambda\rangle=\langle\rho,\lambda\rangle$ and $P_{\tilde\lambda}(x;t)=P_\lambda(x^{-1};t)$. On the Hermitian side $\check C_\lambda=C_{\tilde\lambda}$, and the sign in \eqref{eq:exactminus} satisfies $\sigma(\tilde\lambda)=(-1)^{n(\lambda)}=(-1)^{(n+1)|\lambda|}\sigma(\lambda)$. Hence $\cH(\check C_\lambda^\tw)(x)=\cH(C_\lambda)(x^{-1})$. So $(\check\Phi,\check\Psi^\tw)$ is the Hironaka pair of $P(x^{-1})$.

Finally, if $P$ is homogeneous of degree $d$ then $P(q^{-n}x)=q^{-nd}P(x)$; together with Lemma~\ref{lem:graded} this gives stability under $\iota$.
\end{proof}

\section{Fourier transform on spherical transforms}

\subsection{Fourier transforms}
Jacquet's Fourier transforms \cite{J03}, \cite[\S4]{J2} are
\[
\hat\Phi(X)=\int_{M_n(F)}\Phi(Y)\psi(\Tr(-wXwY))\,dY,\qquad
\hat\Psi(X)=\int_{V_n}\Psi(Y)\psi(\Tr(-wXwY))\,dY,
\]
with self-dual measures. On bi-$K'$-invariant functions $\hat\Phi$ coincides with the Godement--Jacquet transform $\int\Phi(Y)\psi(\Tr XY)dY$, because $-w\in K'$. On $K$-invariant functions $\hat\Psi$ coincides with Hironaka's transform $\int\Psi(Y)\psi(\Tr XY)dY$ \cite[(1.2)]{H98}: choose $u\in\cO_E^\times$ with $u\bar u=-1$; then $-wXw=k\cdot X$ for $k=uw\in K$. In all three normalizations the characteristic functions $\Phi_0$ of $M_n(\cO_F)$ and $\Psi_0$ of $V_n\cap M_n(\cO_E)$ are self-dual. We write $\cF$ for these transforms.

For $\Psi\in H_S$, extend $\Psi^{\mathrm{tw}}$ by zero to $V_n$.
This extension is Schwartz. Define the twisted Fourier transform
as a function on $S$ by
\[
\mathcal F^{\mathrm{tw}}\Psi(X)
=\omega(\det X)^{n+1}\,
\mathcal F(\Psi^{\mathrm{tw}})(X),
\qquad X\in S.
\]
Although $\mathcal F(\Psi^{\mathrm{tw}})$ is Schwartz on $V_n$,
the outer determinant twist need not have a Schwartz extension
across the singular locus. Nevertheless, for every $m\in\mathbb Z$,
\[
(\mathcal F^{\mathrm{tw}}\Psi)^{(m)}
=
\bigl((\mathcal F(\Psi^{\mathrm{tw}}))^{(m)}\bigr)^{\mathrm{tw}}
\in H_S.
\]
All subsequent uses of $\mathcal F^{\mathrm{tw}}\Psi$ concern
its values on $S$ or these determinant shells.

Throughout this section
\[
a=q^{-\frac{n+1}2},\qquad b=q^{\frac{n-1}2},\qquad G(x)=\prod_{i=1}^n\frac{1-a\,x_i^{-1}}{1-b\,x_i},\qquad (\iota P)(x)=P(q^{-n}x^{-1}).
\]
Expanding each factor $(1-bx_i)^{-1}=\sum_{k\ge0}b^kx_i^k$ at the origin gives
\begin{equation}\label{eq:Dj}
G(Tx)=\sum_{j\ge-n}D_j(x)\,T^j,\qquad D_j=\sum_{r=0}^n(-a)^r\,b^{r+j}\,e_r(x^{-1})\,h_{r+j}(x),
\end{equation}
where $h_k$ are the complete symmetric polynomials ($h_k=0$ for $k<0$); $D_j\in\Lambda$ is homogeneous of degree $j$. For a Laurent polynomial $Q$ let $[Q]_k$ denote its homogeneous component of degree $k$. For $P\in\Lambda$ and $m\in\Z$ define
\begin{equation}\label{eq:Gm}
[G\cdot\iota P]_m=\sum_{j\ge-n}D_j\,[\iota P]_{m-j}\ \in\Lambda .
\end{equation}
The sum is finite because $\iota P$ is a Laurent polynomial. Equivalently, $[G\cdot\iota P]_m(x)$ is the coefficient of $T^m$ in the Laurent expansion at $T=0$ of the rational function $T\mapsto G(Tx)\,(\iota P)(Tx)$; this is what we mean by the degree-$m$ component of $G\cdot\iota P$ \emph{expanded at the origin}. If $P$ is homogeneous of degree $d$, then $\iota P$ is homogeneous of degree $-d$ and
\begin{equation}\label{eq:Gmhom}
[G\cdot\iota P]_m=D_{m+d}\,\iota P .
\end{equation}

The two propositions below are statements about shells. Recall that for a bi-$K'$-invariant, respectively $K$-invariant, Schwartz function $f$, each shell $f^{(m)}$ lies in $\HG$, respectively $\HS$.

\begin{proposition}\label{prop:GL}
Let $\Phi\in\HG$. For every $m\in\Z$,
\[
\cS\big((\cF\Phi)^{(m)}\big)=\big[G\cdot\iota(\cS\Phi)\big]_m .
\]
\end{proposition}

\begin{proof}
Let $\omega_x$ be the zonal spherical function on $G'$ with $\int c_\lambda\,\omega_x=\cS(c_\lambda)(x)$, and let $\pi_x$ be the unramified representation whose spherical function is $\omega_x$. Then $g\mapsto\omega_x(g^{-1})$ is the spherical function of the contragredient and equals $\omega_{x^{-1}}$. For a bi-$K'$-invariant Schwartz function $f$ on $M_n(F)$ put
\[
Z(f,s,x)=\int_{G'}f(g)\,\omega_x(g)\,|\det g|^{s+\frac{n-1}2}\,dg .
\]
For $x$ on the unitary torus $(S^1)^n$ this converges absolutely for $\mathrm{Re}\,s$ large, and since $\int f^{(m)}\omega_x=\cS(f^{(m)})(x)$ is homogeneous of degree $m$,
\begin{equation}\label{eq:Zshell}
Z(f,s,x)=\sum_m T^m\,\cS(f^{(m)})(x)=\cS f(Tx),\qquad T=q^{-s-\frac{n-1}2}.
\end{equation}
The local functional equation of Godement and Jacquet \cite[Theorem 3.3]{GJ}, applied with $\pi=\pi_{x^{-1}}$ and $s$ replaced by $1-s$, reads
\[
Z(\cF f,s,x)=\gamma(1-s,\pi_{x^{-1}},\psi)\,Z(f,1-s,x^{-1}).
\]
Put $u=Tx$. By \eqref{eq:Zshell}, $Z(\cF f,s,x)=\cS(\cF f)(u)$ and $Z(f,1-s,x^{-1})=\cS f(q^{s-\frac{n+1}2}x^{-1})=\cS f(q^{-n}u^{-1})$, the latter a finite sum when $f=\Phi\in\HG$. The factor $\gamma(1-s,\pi_{x^{-1}},\psi)$ depends only on the unramified representation $\pi_{x^{-1}}\otimes|\det|^{1-s}$, hence only on $u$; write it $g(u)$. Thus
\begin{equation}\label{eq:GJmero}
Z(\cF f,s,x)=g(u)\,\cS f(q^{-n}u^{-1})
\end{equation}
as meromorphic functions of $s$, the left side being the meromorphic continuation of the zeta integral. Taking $f=\Phi_0=\cF\Phi_0$, for which $Z(\Phi_0,s,x)=\sum_kb^kh_k(u)=\prod_i(1-bu_i)^{-1}$ by Lemma~\ref{lem:shells} and \eqref{eq:Zshell}, gives $g(u)=G(u)$, since $bq^{-n}=a$.

The identity \eqref{eq:GJmero} does not by itself determine the shells, since its right side has several Laurent expansions in $T$. Fix $x$ on the unitary torus and let $f=\cF\Phi$. For $|T|$ small, by \eqref{eq:Zshell} the left side of \eqref{eq:GJmero} is the absolutely convergent series $\sum_mT^m\cS(f^{(m)})(x)$. The right side is $G(Tx)\,(\iota\cS\Phi)(Tx)$: a Laurent polynomial in $T$ times a function holomorphic on $0<|T|<1/b$, which equals its geometric-series expansion there. By uniqueness of Laurent coefficients on this punctured disc, $\cS(f^{(m)})(x)$ is the coefficient of $T^m$ in the expansion at $T=0$, that is, $[G\cdot\iota(\cS\Phi)]_m(x)$. Both sides are Laurent polynomials in $x$ that agree on the torus, which is Zariski dense; hence they agree identically.
\end{proof}

\begin{proposition}\label{prop:herm}
Let $\Psi\in\HS$. For every $m\in\Z$,
\[
\cH\big((\cF^\tw\Psi)^{(m)}\big)=\big[G\cdot\iota(\cH\Psi)\big]_m .
\]
\end{proposition}

\begin{proof}
For a $K$-invariant Schwartz function $\varphi$ on $V_n$, Hironaka's zeta function \cite{H98} is
\[
Z(\varphi;s)=\int_{\Omega_n}\varphi(x)\prod_{i=1}^n|d_i(x)|^{s_i}\,d\mu(x)=\int_S\varphi(x)\,\omega(x;s)\,d\mu(x),
\]
where $\Omega_n=\{x\in V_n:\prod d_i(x)\ne0\}$ and $d\mu$ is a $G$-invariant measure; the second equality uses $K$-invariance. It converges absolutely on $\mathcal R=\{\mathrm{Re}\,s_i>0\ (i<n),\ \mathrm{Re}\,s_n>n\}$. Rescaling $d\mu$ multiplies everything below by the same constant, so we take $d\mu=dx$.

\emph{Step 1: the functional equation in Hironaka's variables.} By \cite[Theorem I]{H98} with $\chi=1$ and $\psi$ of conductor $\cO_F$,
\[
Z(\hat\varphi;s)=\Gamma_1(s)\,Z^*(\varphi;\chi_*^{n-1};s^*),\qquad s^*=(s_{n-1},\dots,s_1,-s_1-\dots-s_n+n),
\]
where $\chi_*(t)=(-1)^{v(t)}$ and $Z^*$ is defined with the lower right minors $d_i^*$. For $K$-invariant $\varphi$, the substitution $x\mapsto wxw$ ($w\in K$, $d_i^*(wxw)=d_i(x)$) gives $Z^*=Z$. Writing
\[
\chi_*^{n-1}(\det x)=|\det x|^{(n-1)\pi\sqrt{-1}/\log q}
\]
and absorbing this factor into $s^*_n$, a direct computation with \eqref{eq:sz} shows that the resulting spectral parameter satisfies
\[
z^*_i=-z_{n+1-i}-n\pmod{\tfrac{2\pi\sqrt{-1}}{\log q}},\qquad\text{i.e. } y\mapsto q^{-n}y^{-1}\ \text{as multisets, where } y_i=q^{z_i}.
\]
Hironaka gives $\Gamma_1$ in these variables as
\[
\Gamma_1=\prod_{i=1}^n\frac{1+q^{-z_i-\frac{n+1}2}}{1+q^{z_i+\frac{n-1}2}}=\prod_{i=1}^n\frac{1+a\,y_i^{-1}}{1+b\,y_i}.
\]

\emph{Step 2: reduction to symmetric functions.} For $\varphi\in\HS$ put $A_\varphi(y)=\int_S\varphi(x)\Psi_z(x)\,dx$; by \cite[(2.6)]{H99} this is a symmetric Laurent polynomial in $y$, and $Z(\varphi;z)=\omega(1_n;z)A_\varphi(y)$. Extend $A$ shell by shell to $K$-invariant Schwartz $\varphi$. By \cite[(2.5)]{H99},
\[
\omega(1_n;z)=c_n\prod_{1\le i<j\le n}\frac{y_j-q^{-1}y_i}{y_j+y_i},\qquad c_n\ne0,
\]
which is invariant under $y\mapsto q^{-n}y^{-1}$ with reversed order, and homogeneous of degree $0$. Hence, for $\varphi\in\HS$,
\begin{equation}\label{eq:AFE}
\frac{Z(\hat\varphi;z)}{\omega(1_n;z)}=\Gamma_1(y)\,A_\varphi(q^{-n}y^{-1})
\end{equation}
as meromorphic functions, the left side being the meromorphic continuation of the zeta integral of the Schwartz function $\hat\varphi$.

\emph{Step 3: comparison with $\cH$.} By \cite[(2.6), (2.7)]{H99},
\[
A_{C_\lambda}(y)
 =v(K\cdot\vp^\lambda)\Psi_z(\vp^\lambda)
 =(-1)^{n(\lambda)+|\lambda|}
   q^{\langle\rho,\lambda\rangle}
   P_\lambda(y;-q^{-1}).
\]
Comparing with \eqref{eq:exactminus} gives, for all $\varphi$,
\begin{equation}\label{eq:AH}
A_\varphi(y)=\cH(\varphi^\tw)(-y).
\end{equation}
Applying \eqref{eq:AFE} to $\varphi=\Psi^\tw$, substituting $y=-x$, and using \eqref{eq:AH} together with $\Gamma_1(-x)=G(x)$, the right side of \eqref{eq:AFE} becomes $G(x)\,(\iota\cH\Psi)(x)$. As in the proof of Proposition~\ref{prop:GL}, this meromorphic identity does not by itself determine the shells.

\emph{Step 4: shells.} Let $\varphi=\cF(\Psi^\tw)$. It is Schwartz, so $v(\det)\ge m_0$ on its support, and each $\varphi^{(m)}$ lies in $\HS$. On $\mathcal R$, absolute convergence applied to $|\varphi|$ gives $Z(\varphi;s)=\sum_{m\ge m_0}Z(\varphi^{(m)};s)$. Fix $z^0$ with $\mathrm{Re}(z^0_{i+1}-z^0_i)>1$ for $i<n$, so that $\mathrm{Re}\,s_i>0$ for $i<n$, and consider $z=z^0+\tau(1,\dots,1)$, i.e. $y=Ty^0$ with $T=q^\tau$. Along this line only $s_n$ varies, and $\mathrm{Re}\,s_n>n$ means $|T|<R$ for some $R>0$. Since $\omega(1_n;z)$ is homogeneous of degree $0$ and nonzero at $y^0$,
\[
\sum_{m\ge m_0}T^m A_{\varphi^{(m)}}(y^0)=\Gamma_1(Ty^0)\,A_{\Psi^\tw}(q^{-n}T^{-1}(y^0)^{-1})\qquad(0<|T|<R),
\]
the left side converging absolutely. The right side is a Laurent polynomial in $T$ times $\Gamma_1(Ty^0)$, which is holomorphic for $0<|T|<1/(b\max_i|y^0_i|)$ and there equals its geometric-series expansion. By uniqueness of Laurent coefficients on a common punctured disc, $A_{\varphi^{(m)}}(y^0)$ is the degree-$m$ component of the origin expansion of $\Gamma_1\cdot\iota A_{\Psi^\tw}$ at $y^0$. Both are Laurent polynomials in $y$ agreeing on a non-empty open set, hence identically. Finally, $(\cF^\tw\Psi)^{(m)}=(\varphi^{(m)})^\tw$, so by \eqref{eq:AH} $A_{\varphi^{(m)}}(-x)=\cH\big((\cF^\tw\Psi)^{(m)}\big)(x)$, while the substitution $y=-x$ carries the origin expansion of $\Gamma_1\cdot\iota A_{\Psi^\tw}$ to that of $G\cdot\iota(\cH\Psi)$, degree by degree. This gives $\cH\big((\cF^\tw\Psi)^{(m)}\big)=[G\cdot\iota(\cH\Psi)]_m$.
\end{proof}

\begin{remark}\label{rem:consistency}
Applying Hironaka's functional equation from Step 1 directly
to the self-dual Schwartz function $\Psi_0$ gives
\[
\Gamma_1(y)
=\frac{A_{\Psi_0}(y)}
{A_{\Psi_0}(q^{-n}y^{-1})},
\]
where $A_{\Psi_0}$ is interpreted meromorphically. This is confirmed independently by Lemma~\ref{lem:shells}, which gives $A_{\Psi_0}(y)=\prod_i(1+by_i)^{-1}$. In particular, given that the functional equation of \cite{H98} is scalar on $K$-invariant functions, its gamma factor is forced.
\end{remark}

\begin{lemma}\label{lem:shells}
Let $\Phi_0^{(k)}$ and $\Psi_0^{(k)}$ be the restrictions of $\Phi_0$ and $\Psi_0$ to $\{v(\det)=k\}$. Then
\[
\cS(\Phi_0^{(k)})=q^{\frac{n-1}2k}h_k(x),\qquad \cH(\Psi_0^{(k)})=(-1)^{(n+1)k}q^{\frac{n-1}2k}h_k(x).
\]
\end{lemma}

\begin{proof}
$\Phi_0^{(k)}=\sum_{\lambda\ge0,|\lambda|=k}c_\lambda$ and $\Psi_0^{(k)}=\sum_{\lambda\ge0,|\lambda|=k}C_\lambda$. Write $q^{\langle\rho,\lambda\rangle}=q^{\frac{n-1}2k}q^{-n(\lambda)}$ and use $h_k=\sum_{|\lambda|=k}t^{n(\lambda)}P_\lambda(x;t)$ \cite[III, (2.10)]{Mac} with $t=q^{-1}$, respectively, together with \eqref{eq:exactminus}, with $t=-q^{-1}$.
\end{proof}

\section{Fourier transform preserves matching}

\begin{proposition}\label{prop:fourier}
	Let $\Phi\in H_{G'}$ and $\Psi\in H_S$, viewed by extension
	by zero as Schwartz functions on $M_n(F)$ and $V_n$,
	respectively. If $\Phi$ and $\Psi$ $1$-match, then
	\[
	\Omega'[\mathcal F\Phi,\psi;a]
	=
	\gamma_1(a)\,
	\Omega[\mathcal F^{\mathrm{tw}}\Psi,\psi;a]
	\]
	for every invertible diagonal matrix $a$.
\end{proposition}

\begin{proof}
Define normalized orbital integrals
\[
\tilde\Omega'[\Phi,\psi;a]=\prod_{i=1}^{n-1}|a_1\cdots a_i|\,\Omega'[\Phi,\psi;a],\qquad
\tilde\Omega[\Psi,\psi;a]=\prod_{i=1}^{n-1}\omega(a_1\cdots a_i)|a_1\cdots a_i|\,\Omega[\Psi,\psi;a],
\]
so that matching in Jacquet's sense means $\tilde\Omega'[\Phi,\psi;\cdot]=\tilde\Omega[\Psi,\psi;\cdot]$. Jacquet's theorem \cite{J03} (see \cite[(31), (32)]{J2}) states that a single operator, the Kloosterman transform $K_{n,\psi}$, satisfies
\[
K_{n,\psi}\,\tilde\Omega'[\Phi,\psi;\cdot]=\tilde\Omega'[\hat\Phi,\bar\psi;\cdot],\qquad K_{n,\psi}\,\tilde\Omega[\Psi,\psi;\cdot]=\tilde\Omega[\hat\Psi,\bar\psi;\cdot].
\]

Since $\Psi\in H_S$, the zero extension of
$\Psi^{\mathrm{tw}}$ is Schwartz on $V_n$. The assumed
$1$-matching and (1) imply that $\Phi$ and
$\Psi^{\mathrm{tw}}$ match in Jacquet's sense on invertible
diagonal matrices.

If $a_n=0$, every matrix occurring in either orbital integral
is singular. Both orbital integrals therefore vanish, because
$\Phi$ and $\Psi^{\mathrm{tw}}$ are supported on invertible
matrices. Thus $\Phi$ and $\Psi^{\mathrm{tw}}$ match in
Jacquet's sense on its full domain.

Jacquet's theorem now shows that $\mathcal F\Phi$ and
$\mathcal F(\Psi^{\mathrm{tw}})$ match in Jacquet's sense
for $\bar\psi$. For the invariant functions under consideration,
conjugation by $\operatorname{diag}(1,-1,1,\ldots)$ shows that
the orbital integrals for $\bar\psi$ and $\psi$ coincide.
Consequently, for invertible diagonal $a$,
\[
\Omega'[\mathcal F\Phi,\psi;a]
=
\gamma_J(a)\,
\Omega[\mathcal F(\Psi^{\mathrm{tw}}),\psi;a].
\]
On such an orbit, the definition of
$\mathcal F^{\mathrm{tw}}$ gives
\[
\Omega[\mathcal F^{\mathrm{tw}}\Psi,\psi;a]
=
\omega(\det a)^{n+1}
\Omega[\mathcal F(\Psi^{\mathrm{tw}}),\psi;a].
\]
Combining this identity with \eqref{eq:gammas} proves the proposition.
\end{proof}

\section{A local proof of Offen's spherical-transfer theorem}\label{sec:proof}

\begin{proposition}\label{prop:closure}
For every $P\in V$ and every $j\in\Z$, both $D_jP$ and $\iota(D_j)P$ lie in $V$.
\end{proposition}

\begin{proof}
By Lemma~\ref{lem:graded} we may assume $P$ homogeneous of degree $d$. Let $(\Phi,\Psi)$ be its Hironaka pair; these functions are compactly supported in $G'$ and $S$, hence Schwartz on $M_n(F)$ and $V_n$. By Proposition \ref{prop:fourier}, $\mathcal F\Phi$ and
$\mathcal F^{\mathrm{tw}}\Psi$ $1$-match on invertible
diagonal matrices. Restricting this identity to
$v(\det a)=m$ shows that their shells
\[
(\mathcal F\Phi)^{(m)}\in H_{G'},
\qquad
(\mathcal F^{\mathrm{tw}}\Psi)^{(m)}\in H_S
\]
also $1$-match. The second inclusion follows from the
shellwise definition in Section 4.1. By Propositions~\ref{prop:GL} and~\ref{prop:herm}, both have transform $[G\cdot\iota P]_m$, which equals $D_{m+d}\,\iota P$ by \eqref{eq:Gmhom}. Hence $D_j\,\iota P\in V$ for all $j$ and all $P\in V$. By Lemma~\ref{lem:inversion}, $\iota P\in V$, and applying this to $\iota P$ gives $D_jP\in V$. Applying $\iota$ to $D_j\,\iota P\in V$ gives $\iota(D_j)P\in V$.

\end{proof}

By Lemma~\ref{lem:unit} and Proposition~\ref{prop:closure}, $V$ contains the subalgebra $\cA\subset\Lambda$ generated by all $D_j$ and $\iota(D_j)$.

\begin{proposition}\label{prop:generation}
$\cA=\Lambda$.
\end{proposition}

\begin{proof}
First, $D_{-n}=(-a)^ne_n^{-1}$ and $\iota(D_{-n})=(-a)^nq^{n^2}e_n$, so $e_n^{\pm1}\in\cA$. Next,
\[
G(Tx)=\prod_{i=1}^n\frac{-a\,T^{-1}x_i^{-1}\,(1-a^{-1}Tx_i)}{1-bTx_i}=(-a)^nT^{-n}e_n^{-1}\prod_{i=1}^n\frac{1-q^{\frac{n+1}2}Tx_i}{1-q^{\frac{n-1}2}Tx_i}.
\]
Let $q_r(x;t)$ be the Hall--Littlewood functions \cite[III, (2.9)]{Mac},
\[
\sum_{r\ge0}q_r(x;t)\,u^r=\prod_{i=1}^n\frac{1-tx_iu}{1-x_iu}.
\]
With $u=bT$ and $t=a^{-1}b^{-1}=q$ we obtain $e_n\,D_{r-n}=(-a)^nb^r\,q_r(x;q)$, so $q_r(x;q)\in\cA$ for all $r\ge0$. Taking logarithms,
\[
\log\sum_{r\ge0}q_r(x;q)u^r=\sum_{r\ge1}\frac{1-q^r}{r}\,p_r(x)\,u^r,
\]
where $p_r$ are the power sums. Since $1-q^r\ne0$ for $r\ge1$, each $p_r$ is a polynomial in $q_1,\dots,q_r$, and conversely. Hence $\cA$ contains $\C[p_1,\dots,p_n]=\C[e_1,\dots,e_n]$, and with $e_n^{-1}$ it contains $\Lambda$.
\end{proof}

\begin{proof}[Local proof of Offen's theorem]
Take $\epsilon=-1$, so $\delta_\epsilon=1$. By Lemma~\ref{lem:unit} and Propositions~\ref{prop:closure} and~\ref{prop:generation}, $V\supset\cA=\Lambda$. Since $\cS$ and $\cH$ are isomorphisms onto $\Lambda$, every pair $(\Phi,\Psi)\in\HG\times\HS$ with $\cS\Phi=\cH\Psi$ is a Hironaka pair with transform in $V$, hence $1$-matches. This proves the theorem locally in the $\epsilon=-1$ normalization, and Lemma~\ref{lem:change} gives the equivalent $\epsilon=+1$ form, with $\delta_{+1}=0$.
\end{proof}

\begin{corollary}[Jacquet]\label{cor:jacquet}
For every $f\in\cH_G$, the functions $X\mapsto\int_G\mathbf 1_{K\cap S}({}^t\bar gXg)f(g)\,dg$ and $X\mapsto\int_{G'}\mathbf 1_{K'}(Xg)\,b(f)(g)\,dg$ match in Jacquet's sense.
\end{corollary}

\begin{proof}
As observed in \cite{O06}, both functions have the same transform, namely the Satake transform of $f$ (or of $\check f$, according to the convolution convention) evaluated at $x^2$, by the $\cH_G$-linearity of $\cH$ \cite[Theorem 2]{H99} and the definition of base change. Their supports lie in $\{v(\det)\in2\Z\}$, where $\gamma_1=\gamma_J$ by \eqref{eq:gammas}.
\end{proof}

\end{document}